\documentclass[11pt]{amsart}
\usepackage{fullpage}
\usepackage[utf8]{inputenc}
\usepackage{amsmath,amsthm,amssymb,amscd}
\usepackage[all,cmtip]{xy}
\usepackage{dsfont}
\usepackage{mathrsfs}
\usepackage{leftidx}

\usepackage[square,sort,comma,numbers]{natbib}

\usepackage{color}
\usepackage{hyperref}
\hypersetup{
   colorlinks=true,
   linktoc=all,
   linkcolor=black,
   citecolor=black,
   urlcolor=black,
}

\newtheorem{thm}{Theorem}[section]
\newtheorem{cor}[thm]{Corollary}
\newtheorem{lem}[equation]{Lemma}
\newtheorem{prop}[equation]{Proposition}

\newtheorem{conjecture}{Conjecture}[section]

\theoremstyle{definition}
\newtheorem{defn}{Definition}[section]

\theoremstyle{remark}

\numberwithin{equation}{section}

\newcommand{\bbA}{\mathbb{A}}

\newcommand{\bbC}{\mathbb{C}}

\newcommand{\bbF}{\mathbb{F}}
\newcommand{\bbG}{\mathbb{G}}

\newcommand{\bbQ}{\mathbb{Q}}

\newcommand{\cA}{\mathcal{A}}

\newcommand{\cE}{\mathcal{E}}

\newcommand{\cH}{\mathcal{H}}

\newcommand{\cL}{\mathcal{L}}
\newcommand{\cM}{\mathcal{M}}
\newcommand{\cN}{\mathcal{N}}
\newcommand{\cO}{\mathcal{O}}
\newcommand{\cP}{\mathcal{P}}

\newcommand{\cX}{\mathcal{X}}

\newcommand{\fC}{\mathfrak{C}}
\newcommand{\fD}{\mathfrak{D}}

\newcommand{\fc}{\mathfrak{c}}

\newcommand{\fg}{\mathfrak{g}}

\newcommand{\ft}{\mathfrak{t}}

\newcommand{\ga}{\gamma}
\newcommand{\la}{\lambda}

\newcommand{\spec}{\mathrm{Spec}}

\newcommand\bm[1]{\begin{bmatrix}#1\end{bmatrix}}

\newcommand{\ving}{\mathrm{Vin}_{G}}

\title{An overview of the geometry of Kottwitz-Viehmann varieties}
\author{Jingren Chi}
\address{Morningside Center of Mathematics and State Key Laboratory of Mathematical Sciences, Academy of Mathematics and Systems Science, Chinese Academy of Sciences, Beijing 100190, China}
\email{jrenchi@amss.ac.cn}
\date{}

\begin{document}
\begin{abstract}
    We give a survey on the work \cite{Chi} about geometrization of orbital integrals of spherical Hecke functions on reductive groups over non-archimedean local fields.\footnotemark
\end{abstract}
\maketitle
\stepcounter{footnote}
\footnotetext{This is an updated version of an expository article appeared in \emph{Proceedings of the International Consortuim of Chinese Mathematicians 2019, vol.2}. Compared to the published version, we add Section \ref{sec:example} about an example in $\mathrm{SL}_3$.}

\section{Introduction}
\subsection{Motivations}
The Arthur-Selberg trace formula and its variants are important tools in automorphic representation theory and number theory. On the geometric side of the trace formula, one encounters orbital integrals on a reductive group over Archimedean or $p$-adic local fields. Most applications of trace formulas involve detailed study of these orbital integrals. In practice, it seems that the study of non-Archimedean orbital integrals are more challenging than the Archimedean ones.\par 
Our general motivation is to apply algebro-geometric methods to understand certain $p$-adic orbital integrals. In fact, currently this method has been more fruitful in the analogous situation of equal-characteristic non-Archimedean local field. In this paper we focus on a specific class of such orbital integrals.\par 
To be more precise, we consider a split connected reductive group $G$ over a non-Archimedean local field $F$ with residue field $k=\bbF_q$. Let $\cO$ be the ring of integers in $F$ and $\varpi\in\cO$ a fixed uniformizer.  Denote $K=G(\cO)$. Let $T$ be a split maximal torus of $G$ and $B$ a Borel subgroup containing $T$. Let $X_*(T)$ be the coweight lattice of $T$ and $X_*(T)_+$ the set of dominant coweights (with respect to the simple roots defined by $B$). For each regular semisimple element $\ga\in G(F)^{\mathrm{rs}}$, let $G_\ga$ be its centralizer. Then the orbital integral of a test function $f\in C_c^\infty(G(F))$ along the adjoint orbit of $\ga$ is 
\[O_\ga(f):=\int_{G_\ga(F)\backslash G(F)}f(g^{-1}\ga g)\frac{dg}{dg_\ga}\]
where $dg$ is a Haar measure on $G(F)$ and $dg_\ga$ is a Haar measure on $G_\ga(F)$ and we integrate against the quotient measure on $G(F)/G_\ga(F)$. The specific class of functions we are interested in are those in the spherical Hecke algebra $\cH_K$, consisting of locally constant $\bbC$-valued functions on $G(F)$ that are compactly supported and bi-invariant under $K$. By Cartan decomposition, $\cH_K$ has a basis consisting of characteristic functions of double cosets $K\varpi^\la K$ for $\la\in X_*(T)_+$, where $\varpi^\la:=\la(\varpi)\in T(F)$. So the basic building blocks are the orbital integrals $O_\ga(1_{K\varpi^\la K})$. Unwinding the definitions, we see that, in a suitable sense, this integral counts the quotient groupoid $[X_\ga^\la/G_\ga(F)]$ where
\[X_\ga^\la:=\{g\in G(F)/K | g^{-1}\ga g\in K\varpi^\la K\}.\]
The sets $X_\ga^\la$ have been studied in the work of Kottwitz-Viehmann \cite{KoV}. In \cite{Bou15} and \cite{Chi} it is shown that they are $k$ points of certain finite dimensional algebraic varieties locally of finite type over $k$ (this fact is also proved in a previous work of Lusztig \cite{Lu15}). We refer to them as ``\emph{Kottwitz-Viehmann varieties}". We would get the asymptotic behaviour of the orbital integrals $O_\ga(1_{K\varpi^\la K})$ through the study of basic geometric properties of the Kottwitz-Viehmann varieties $X_\ga^\la$

\subsection{Plan of the paper}
Our aim is to give an expository account of some recent works on the geometry of Kottwitz-Viehmann varieties, including \cite{Bou15}, \cite{BC18} and \cite{Chi}. We start by some generalities on 
orbital integrals and their geometric interpretation in \S\ref{sec:orbital-integral}. In \S\ref{sec:statement} we state the results on non-emptiness pattern and dimension formula for Kottwitz-Viehmann varieties. Then in \S\ref{sec:ASF} we review the analogous Lie algebra situation where one encounters affine Springer fibers and compare them with Kottwitz-Viehmann varieties. In \S\ref{sec:Vinberg}, we give a brief review of the theory of Vinberg monoids that is used in the study of Kottwitz-Viehmann varieties. In \S\ref{sec:Hitchin}, we introduce the global analogue of Kottwitz-Viehmann varieties and discuss its role in the proof of the dimension formula. The last section \S7 is devoted to studying an example in $\mathrm{SL}_3$ case.

\subsection{Notations}
$G$ always denotes a (split) connected reductive group and we always assume that the characteristic of the base field does not divide the order of the Weyl group $W$ of $G$. Let $\fg:=\mathrm{Lie}(G)$ be the Lie algebra of $G$.\par 
Fix a maximal torus $T$ of $G$ and a Borel subgroup $B$ containing $T$. Let $\Delta=\{\alpha_1,\dotsc,\alpha_r\}$ be the set of simple roots determined by $T\subset B$. Let $\check{\Lambda}:=X^*(T)$ (resp. $\Lambda:=X_*(T)$) be the weight (resp. coweight) lattice. Let $\check{\Lambda}^+$ (resp. $\Lambda^+$) be the set of dominant weights (resp. dominant coweights). Let $W$ be the Weyl group of $G$ and $S\subset W$ the set of simple reflections associated to the simple roots $\Delta$. There is a unique longest element $w_0$ of $W$ under the Bruhat order determined by $S$.

\subsection*{Acknowledgement}
This survey article is based on the author's talk at the 8th International Congress of Chinese Mathematicians in 2019. I thank the organizers of the conference for the invitation and financial support. This work is started when I was a postdoc at Départment de Mathématiques d'Orsay. I thank Fondations Mathématiques Jacquet Hadamard for the financial support.

\section{Generalities on non-archimedean orbital integrals}\label{sec:orbital-integral}
In this section we let $F$ be a non-archimedean local field with ring of integers $\cO$ and residue field $\kappa=\bbF_q$, the finite field of characteristic $p$ with $q$ elements. Then $F$ is a finite extension of either $\bbQ_p$ or $\bbF_q((t))$. \par 
Our interest is in orbital integrals of functions in the spherical Hecke algebra $\cH_K$ consisting of locally constant $\bbC$-valued functions with compact support that are left and right invariant under $K=G(\cO)$. Recall that we have the Cartan decomposition 
\[G(F)=\bigsqcup_{\la\in X_*(T)_+}K\varpi^\la K\]
which implies that the characteristic functions $1_{K\varpi^\la K}$ form a basis of $\cH_K$. For this reason, we focus on orbital integrals of $1_{K\varpi^\la K}$.\par 
Let $\ga\in G(F)$ be a strongly regular semisimple element. Then its centralizer $G_\ga$ is a maximal torus defined over $F$. We fix the Haar measure $dg$ on $G(F)$ so that $K=G(\cO)$ has volume one. Also fix a Haar measure $dg_\ga$ on $G_\ga(F)$. Let $d\dot{g}=\frac{dg}{dg_\ga}$ be the quotient measure on $G_\ga(F)\backslash G(F)$. The orbital integral we study is defined as
\[O_\ga(1_{K\varpi^\la K})=\int_{G_\ga(F)\backslash G(F)}1_{K\varpi^\la K}(g^{-1}\ga g)d\dot{g}.\]
We can interpret $O_\ga(1_{K\varpi^\la K})$ as a weighted counting as follows. Consider the set
\[X_\ga^\la=\{g\in G(F)/K | g^{-1}\ga g\in K\varpi^\la K\}\]
The centralizer $G_\ga(F)$ acts naturally on $X_\ga^\la$ by left multiplication on right $K$-cosets. For each $x\in X_\ga^\la$, let $\mathrm{Stab}_{G_\ga(F)}(x)$ be its stabilizer in $G_\ga(F)$, which is a compact open subgroup of $G_\ga(F)$. Unwinding the definition, we see that
\[O_\ga(1_{K\varpi^\la K})=\sum_{x\in G_\ga(F)\backslash X_\ga^\la}\frac{1}{\mathrm{vol}(\mathrm{Stab}_{G_\ga(F)}(x),dg_\ga)}\]
where $x$ ranges over a set of representatives for the $G_\ga(F)$-orbits on $X_\ga^\la$.\par 
Among various general goals for studying orbital integrals, we focus on their asymptotic estimates. The basic observation is that $X_\ga^\la$ has an algebro-geometric structure so that the right hand side above could be interpreted roughly as counting $\bbF_q$-points of certain algebraic varieties. To be more precise, in the case where $F$ is equi-characteristic, $X_\ga^\la$ is the set of $\bbF_q$-points of a locally closed subscheme of the affine Grassmanian for $G$; in the case where $F$ is mixed-characteristic, $X_\ga^\la$ is a perfect algebraic space over $\bbF_q$, which is a locally closed subscheme of the mixed-characteristic affine Grassmanian. 
We denote the underlying scheme (or perfect algebraic space) also by $X_\ga^\la$. Our primary interest is to obtain an asymptotic estimate of
$O_\ga(1_{K\varpi^\la K})$. This is related to the geometry of $X_\ga^\la$ as follows. For each finite extension $\bbF_{q^n}$ of $\bbF_q$ of degree $n$, let $F_n$ be the corresponding unramified extension of $F$ and $\cO_n$ its ring of integers. Let $K_n:=G(\cO_n)$. We can view $\ga$ as an element in $G(F_n)$ and consider the orbital integral of the function $1_{K_n\varpi^\la K_n}$, now a basis vector in the spherical Hecke algebra of $G(F_n)$. Then $O_\ga(1_{K_n\varpi^\la K_n})$ can be expressed as a polynomial in $q^n$ and as $n\to\infty$, the highest degree of this polynomial will be the dimension of $X_\ga^\la$. In other words, we have the asymptotic estimate
\[O_\ga(1_{K_n\varpi^\la K_n})=O(q^{n\dim X_\ga^\la}),\quad\text{ as }n\to\infty\]
Moreover, for $n$ large enough, the coefficient of the leading term $q^{n\dim X_\ga^\la}$
will be, up to a volume factor, the number of top dimensional irreducible components of the quotient stack $[X_\ga^\la/G_\ga(F)]$.\par 
Our current focus is in the equi-characteristic case. We refer to the schemes $X_\ga^\la$ as \emph{Kottwitz-Viehmann varieties}. The advantage of equi-characteristic situation is that one can obtain natural deformations of $X_\ga^\la$ via global methods, which is currently unavailable in mixed-characteristic case. Since we will be interested in geometric properties (mainly dimension and irreducible components), We assume from now on that the base field $k$ is algebraically closed and let $F=k((\varpi))$ and $\cO=k[[\varpi]]$.

\section{Statement of results}\label{sec:statement}

In this section we state the main results on the geometry of Kottwitz-Viehmann varieties, including non-emptiness pattern, dimension formula, a conjectural description of number of irreducible componenets and partial results on the conjecture. We start by reviewing some invariants of a regular semisimple conjugacy class used in the statements of results. 

\subsection{Invariants of a regular semisimple conjugacy class}
\subsubsection{Newton Point}\label{Newton-point-section}
Following \cite[\S4]{KoV}, for each $\ga\in G(F)^{\mathrm{rs}}$, one associate a rational dominant coweight $\nu_\ga\in X_*(T)^+_\bbQ$, called the \emph{Newton point of $\ga$}. It is defined as follows. There is a totally ramified extension $\tilde{F}$ of $F$ such that $\ga$ is $G(\tilde{F})$ conjugate to an element in $T(\tilde{F})$. Consider the $W$-orbit of its image in $T(\tilde{F})/T(\tilde{\cO})=\frac{1}{[\tilde{F}:F]}X_*(T)$. Here $\tilde{\cO}$ is the ring of integers of $\tilde{F}$. Then the unique dominant element in this $W$-orbit is the Newton point $\nu_\ga$. 

\subsubsection{Kottwitz map}
Let $\pi_1(G)$ be the fundamental group of $G$, defined as the quotient of the coweight lattice by the coroot lattice of $T$. Let $p_G:X_*(T)\to\pi_1(G)$ be the natural projection. Following \cite{KoV}, one defines a group homomorphism
\[\kappa_G: G(F)\to\pi_1(G)\]
which we refer to as Kottwitz map. Note that in \emph{loc. cit.}, this map is denoted by $w_G$. The map $\kappa$ can also be understood more explicitly as follows: let $G(F)_1$ be the subgroup of $G(F)$ generated by all parahoric subgroups, then $\kappa_G$ induces an isomorphism $G(F)/G(F)_1\cong\pi_1(G)$. 

\subsubsection{Discriminant valuation}
The \emph{discriminant valuation} for a regular semisimple element $\ga\in G(F)^{\mathrm{rs}}$ is defined by
\[d(\ga):=\mathrm{val}\det(\mathrm{Id}-\mathrm{ad}_\ga:\fg(F)/\fg_\ga(F)\to\fg(F)/\fg_\ga(F))\]
where $\fg$ is the Lie algebra of $G$ and $\fg_\ga$ is the centralizer of $\ga$, i.e. the fixed locus of the adjoint action $\mathrm{ad}_\ga$. \par 
For example, in the case $G=\mathrm{SL}_n$, $\ga\in G(F)^{\mathrm{rs}}$ is an $n\times n$ matrix of determinant $1$ and $d(\ga)$ is the valuation of the discriminant of the characteristic polynomial of $\ga$. In other words, if $\{\la_i,1\le i\le n\}$ are the (distinct) eigen-values of $\ga$ in an algebraic closure of $F$, then we have
\[d(\ga)=\mathrm{val}\prod_{i\ne j}(\la_i-\la_j).\]

\subsection{Non-emptiness pattern and dimension formula}
Let $\la\in X_*(T)_+$ and $\ga\in G(F)^{\mathrm{rs}}$ be a regular semisimple element. The following theorem gives non-emptiness criterion for the Kottwitz-Viehmann variety $X_\ga^\la$.

\begin{thm}[Non-emptiness pattern]\label{thm:nonempty}
The following are equivalent:
\begin{enumerate}
\item[(i)] $X_{\ga}^\lambda$ is nonempty;
\item[(ii)] $\kappa_G(\ga)=p_G(\la)$ and $\nu_\ga\le_\bbQ\la$, i.e. $\la-\nu_\ga$ is a $\bbQ$-linear combination of simple coroots with non-negative coefficients;
\end{enumerate}
\end{thm}
This is extracted from \cite[Proposition 3.1.6]{Chi}. See also \cite[Th\'eor\`eme 3.10]{Bou15}, which is proved under the assumption that $G$ is semisimple simply-connected.\par 
The implication (i)$\Rightarrow$(ii) is done in \cite[Corollary 3.6]{KoV}. For the other direction one uses the theory of Vinberg monoid. Roughly speaking, one extend the Steinberg section of the adjoint invariant quotient $G\to G//\mathrm{Ad}(G)$ to the Vinberg monoid to produce a point in $X_\ga^\la$.

\begin{thm}[Dimension formula]\label{thm:dim-formula}
Assume that $k$ is algebraically closed and its characteristic does not divide the order of Weyl group of $G$. Then 
$X_\ga^\la$ is a $k$-schemes locally of finite type. When it is non-empty, it is equidimensional of dimension 
\[\dim X_\ga^\la=\dim X_\ga^{\le\la}=\langle\rho,\la\rangle+\frac{1}{2}(d(\ga)-c(\ga))\]
where 
\begin{itemize}
\item $\rho$ is half sum of the positive roots for $G$;
\item $d(\ga)$ is the discriminant valuation of $\ga$;
\item $c(\ga)=\mathrm{rank}(G)-\mathrm{rank}_F(G_\ga)$, the difference between the dimension of the maximal torus of $G$ and the dimension of the maximal $F$-split subtorus of the centralizer $G_\ga$.
\end{itemize}
\end{thm}
This is established in \cite{Bou15} and \cite{BC18} in the case $G$ is semisimple simply-connected and in \cite{Chi} for any split connected reductive group. We remark this result can be viewed as the analogue of the dimension formula for affine Deligne-Lusztig varieties. On the other hand, it can also be viewed as a generalization of the dimension formula for affine Springer fibers established in \cite{KL} and \cite{Be96}. In fact, 
a large part of the proof proceeds by generalizing the methods in \emph{loc. cit.}, with one new ingredient. To be more precise, the proof starts by analyzing the natural action of the loop group $LG_\ga$ on $X_\ga^\la$.  We first study the \emph{regular locus} $X_\ga^{\la,\mathrm{reg}}$, which is the union of the maximal $LG_\ga$-orbits on $X_\ga^\la$. The formal definition of $X_\ga^{\la,\mathrm{reg}}$ will use the theory of Vinberg monoid. Generalizing the methods for affine Springer fibers, one can show that $\dim X_\ga^{\la,\mathrm{reg}}$ is given by the desired formula and it remains to show that $\dim X_\ga^\la=\dim X_\ga^{\la,\mathrm{reg}}$. 
For this step, the methods in \cite{KL} fail to generalize. The failure is related to some interesting phenomena on irreducible components of $X_\ga^\la$ and $X_\ga^{\la,\mathrm{reg}}$ that we discuss in the next section. To finish this last step for dimension formula, we took a new approach using global methods that we discuss in \S\ref{sec:Hitchin}. Currently there is no purely local proof of this step and the method we use does not seem to generalize to mixed-characteristic case.

\subsection{Irreducible components of Kottwitz-Viehmann varieties}\label{sec:irr-components}
In general $X_\ga^{\la}$ may have infinitely many irreducible components since the loop group $G_\ga(F)$ that acts on it may have infinitely many connected components. However, one can show that the quotient stack $[X_\ga^\la/G_\ga(F)]$ always has finitely many irreducible components and we are interested in the number of its components. We have the following conjectural description of this number:
\begin{conjecture}\label{conj:irr-components}
Let $\ga\in G(F)^{\mathrm{rs}}$ and $\la\in X_*(T)_+$. Let $\mu\in X_*(T)_+$ be the minimal dominant integral coweight that dominates the Newton point $\nu_\ga\in X_*(T)_{\bbQ,+}$. Then $\#\mathrm{Irr}[X_\ga^\la/G_\ga(F)]$, or in other words, the number of $G_\ga(F)$-orbits on the set $\mathrm{Irr}(X_\ga^\la)$, equals to the weight multiplicity $m_{\la\mu}$ (which is defined as the dimension of the $\mu$-weight space in the irreducible representation of the Langlands dual group $\hat{G}$ with highest weight $\la$).
\end{conjecture}

\begin{thm}\label{thm:irr-component-hyperbolic}
Conjecture~\ref{conj:irr-components} is true when $\ga$ is hyperbolic, i.e. $G(F)$ conjugate to a strongly regular semisimple element in $T(F)$.
\end{thm}
We give a brief outline of the proof as the argument also proves the dimension formula in the hyperbolic case. We may assume that $\ga\in\varpi^\mu T(\cO)\cap G(F)^{\mathrm{rs}}$ where $\mu\in X_*(T)_+$ is the Newton point of $\ga$. Then $G_\ga=T$ and by the Iwasawa decomposition, $X_\ga^\la$ has a stratification whose strata correspond bijectively to $X_*(T)$ and each stratum is isomorphic to 
\[Y_\ga^\la=\{u\in U(F)/U(\cO)| u^{-1}\ga u\in K\varpi^\la K\}.\]
Here $U$ is the unipotent radical in the Borel $B$. Then we have
\[\dim X_\ga^\la=\dim Y_\ga^\la\]
and there is a natural bijection
\[\mathrm{Irr}(Y_\ga^\la)\xrightarrow{\sim}\mathrm{Irr}([X_\ga^\la/G_\ga(F)]).\]
Consider the map $f_\ga:U(F)\to U(F)$ defined by $f_\ga(u)=u^{-1}\ga u\ga^{-1}$. Then we have the following diagram
\[\xymatrix{ 
f_\ga^{-1}(U(F)\cap K\varpi^\la K\varpi^{-\mu})\ar[r]^{f_\ga}\ar[d] & U(F)\cap K\varpi^\la K\varpi^{-\mu}\ar[d]\\
Y_\ga^\la & S_{\la\mu}
}\]
where $S_{\la\mu}=(U(F)\varpi^\mu K\cap K\varpi^\la K)/K$ is the MV cycle in the affine Grassmanian $\mathrm{Gr}_G=G(F)/K$. The left vertical map is quotient by $U(\cO)$ and the right vertical map sends $g$ to $g\varpi^\mu K/K$. The 3 maps in the diagram are (projective limits of) smooth fibrations with fibers isomorphic to $\bbA^n$ for some $n$. In particular, they induce bijections on the set of irreducible components. We know from geometric Satake theory that $S_{\la\mu}$ has $m_{\la\mu}$ irreducible components, hence $Y_\ga^\la$ and $[X_\ga^\la/G_\ga(F)]$ also have $m_{\la\mu}$ irreducible components. Moreover, by passing to finite type quotients, one can
analyze the relative dimension of the 3 maps and deduce the dimension formula for $Y_\ga^\la$ (and hence $X_\ga^\la$) from the dimension formula for the MV cycle $S_{\la\mu}$. 

\section{Comparison with affine Springer fibers}\label{sec:ASF}
In this section we first recall the geometric properties of affine Springer fibes and then compare them with Kottwitz-Viehmann varieties.
\subsection{Affine Springer fiber in the affine Grassmanian}
For a regular semisimple element $\ga\in\fg^{\mathrm{rs}}(F)$ in the Lie algebra, the affine Springer fiber $X_\ga$ is the reduced closed subscheme of the affine Grassmanian $\mathrm{Gr}_G=LG/L^+G$ whose set of closed point is
\[X_\ga=\{g\in G(F)/G(\cO)| \mathrm{ad}(g)^{-1}(\ga)\in\fg(\cO)\}\]
Let $I\subset G(\cO)$ be the Iwahori subgroup consisting of elements in $G(\cO)$ whose reduction mod $\varpi$ lies in the Borel $B$. We also view $I$ as an affine group scheme (of infinite type) over the residue field $k$. One also studies the affine Springer fiber $Y_\ga$ in the affine flag variety $\mathrm{Fl}_G=LG/I$ whose set of closed point is 
\[Y_\ga=\{g\in G(F)/I | \mathrm{ad}(g)^{-1}(\ga)\in\mathrm{Lie}(I)\}.\]
By definition, there is a natural morphism $p:Y_\ga\to X_\ga$.\par 
In \cite{KL}, $Y_\ga$ is introduced as the affine analogue of the classical Springer fibers and is the main object of study in \emph{loc.cit.} In this article we are more interested in $X_\ga$ since it is the Lie algebra analogue of the Kottwitz-Viehmann varieties. However, to study the geometric properties of $X_\ga$ along the lines in \cite{KL}, one relies heavily on $Y_\ga$.\par 
Let $G_\ga$ be the stabilizer of $\ga\in\fg(F)$ under the adjoint action of $G$. Since $\ga$ is regular semisimple, $G_\ga$ is a maximal torus defined over $F$. As in the case of Kottwitz-Viehmann varieties, the loop group $LG_\ga$ acts naturally on $X_\ga$. When $k$ is an algebraic closure of a finite field, the quotient $[X_\ga/LG_\ga]$ encodes the orbital integral $O_\ga(1_{\fg(\cO)})$, as well as its stable and unstable variants. This is the starting point of Ng\^o's proof of the Langlands-Shelstad fundamental lemma for Lie algebras in the function field setting, cf. \cite{Ngo10}. In the following we review the methods for studying geometric properties of $X_\ga$.

\subsection{Chevalley morphism}
First we recall some standard facts about the Chevalley morphism. The \emph{Chevalley} base is defined as
\[\fc:=\fg//\mathrm{ad}(G)=\spec k[\fg]^G,\] 
the invariant quotient of the Lie algebra under the adjoint action of $G$. Restricting regular functions on $\fg$ to $\ft:=\mathrm{Lie}(T)\subset\fg$, we get an isomorphism $k[\fg]^G\cong k[\ft]^W$. This implies that the algebra $k[\fg]^G$ is isomorphic to a free polynomial algebra in $r$ generators where $r$ is the rank of $G$.\par 
The natural morphism $\chi:\fg\to\fc$ is called the \emph{Chevalley morphism}. The morphism $\chi$ is surjective flat with integral fibers. In each fiber there is a unique open $G$-orbit (the regular conjugacy class) and a unique closed $G$-orbit (the semisimple conjugacy class). The regular locus $\fg^{\mathrm{reg}}$ is the union of the open orbits in each fiber and it is an open subset of $\fg$ whose complement has codimension 3.\par
There is an open subset $\fc^{\mathrm{rs}}\subset\fc$ consisting of regular semisimple conjugacy class. The fiber of $\chi$ over $\fc^{\mathrm{rs}}$ consists of a single conjugacy class, which is both regular and semisimple. There union is the regular semisimple open subset $\fg^{\mathrm{rs}}=\chi^{-1}(\fc^{\mathrm{rs}})$ of $\fg$. The complement $\fD_\fg=\fc\setminus\fc^{\mathrm{rs}}$ is the discriminant divisor, defined as the vanishing loci of the discriminant function
\[D_\fg:=\prod_{\alpha\in\Phi}d\alpha\in k[\ft]^W\cong k[\fg]^G\]
where the product ranges over the set of roots. 

\subsection{Regular centralizer for Lie algebra and Kostant section}
The restriction $\chi^{\mathrm{reg}}:\fg^{\mathrm{reg}}\to\fc$ of the Chevalley morphism to the regular open subset is smooth and surjective. There is a morphism $\epsilon_\fg:\fc\to\fg^{\mathrm{reg}}$, constructed by Kostant, that defines a section of $\chi^{\mathrm{reg}}$. We refer to it as the \emph{Kostant section}. The morphism $G\times\fc\to\fg^{\mathrm{reg}}$ defined by $(g,a)\mapsto\mathrm{ad}(g)\epsilon_\fg(a)$ is smooth surjective and induces an isomorphism $(G\times\fc)/J\cong\fg^{\mathrm{reg}}$. Taking quotients by $G$, we see that $\epsilon_\fg$ induces an isomorphism $BJ\cong[\fg^{\mathrm{reg}}/G]$. In other words, $[\fg^{\mathrm{reg}}/G]$ is a $BJ$-gerbe neutralized by the Kostant section $\epsilon_\fg$.
\par 
Let $I_\fg$ be the group scheme over $\fg$ whose fiber over $\ga$ is the centralizer $G_\ga$. The \emph{regular centralizer} is the group scheme $J:=\epsilon_\fg^*I_\fg$ over $\fc$. One can also define $J$ by descending $I_\fg$ along the smooth morphism $\chi^{\mathrm{reg}}$. The basic fact is that $J$ is a smooth commutative group scheme over $\fc$ and its restriction to $\fc^{\mathrm{rs}}$ is a torus.\par 
Now let us come back to the affine Springer fiber $X_\ga$ for $\ga\in\fg^{\mathrm{rs}}(F)$. The \emph{regular locus} of $X_\ga^{\mathrm{reg}}$ is the open subscheme of $X_\ga$ defined by
\[X_\ga^{\mathrm{reg}}=\{g\in G(F)/G(\cO)|\mathrm{ad}(g)^{-1}\ga\in\fg^{\mathrm{reg}}(\cO)\}.\]
Using the Kostant section, we get the following non-emptiness criterion of $X_\ga$. See \cite[\S4]{Ngo10} for more details.
\begin{prop}
The following are equivalent:
\begin{enumerate}
    \item $X_\ga$ is non-empty;
    \item $X_\ga^{\mathrm{reg}}$ is non-empty;
    \item $\chi(\ga)\in\fc(\cO)$.
\end{enumerate}
\end{prop}

So from now on we assume that $a:=\chi(\ga)\in\fc(\cO)$. Let $J_a$ be the $\cO$-group scheme defined as the pull-back of the regular centralizer $J$ along the morphism $a:\spec\cO\to\fc$. Then $J_a$ is an integral model of $G_\ga$. Consider the quotient $P_a:=\mathrm{Gr}_{J_a}=G_\ga(F)/J_a(\cO)$, which is a finite dimensional commutative $k$-group.

\begin{prop}\label{prop:ASF-reg-torsor}
The natural action of the loop group $G_\ga(F)$ on $X_\ga$ factors through the quotient $P_a$ and the regular locus $X_\ga^{\mathrm{reg}}$ is a torsor under $P_a$. In particular, $\dim X_\ga^{\mathrm{reg}}=\dim P_a$.
\end{prop}
The main result for $X_\ga$ we review is the following
\begin{thm}[\cite{KL},\cite{Be96}]
The affine Springer fiber $X_\ga$, when nonempty, is a locally of finite type scheme of dimension
\[\dim X_\ga=\frac{d_\ga-c_\ga}{2}\]
where
\begin{itemize}
    \item $d_\ga=\mathrm{val}D_\fg(\ga)$ is the discrmininant valuation;
    \item $c_\ga=\dim T-\mathrm{rank}_FG_\ga$. 
\end{itemize}
\end{thm}
This is the analogue of Theorem~\ref{thm:dim-formula}. It follows from the following two results. 
\begin{thm}\cite{Be96}
$\dim X_\ga^{\mathrm{reg}}=\dim P_a=\frac{d_\ga-c_\ga}{2}$
\end{thm}
When $\ga$ is hyperbolic, i.e. $G(F)$ conjugate to an element in $\ft(F)$, we 
have $c_\ga=0$ and it follows from results in \cite[\S5]{KL} that $\dim X_\ga=\frac{d_\ga}{2}$. In general, the dimension formula for $P_a$ is proved in \cite{Be96} by reducing to the case of hyperbolic conjugacy class.

\begin{thm}\label{thm:ASF-dim-equal}
$\dim X_\ga=\dim X_\ga^{\mathrm{reg}}$ and $\dim(X_\ga-X_\ga^{\mathrm{reg}})<\dim X_\ga$.
\end{thm}
\begin{proof}[Outline of proof]
The first statement is first proved in \cite[\S4]{KL}. In \cite[Proposition 3.7.1]{Ngo10}, Ng\^o simplified some part of the arguments of Kazhdan-Lusztig by showing the second statement. The proof proceeds by exploring the natural morphism $p:Y_\ga\to X_\ga$. One first show as in \cite[\S4]{KL} that $Y_\ga$ is equi-dimensional. Note that in \emph{loc.cit.} it is assumed that $G$ is simply connected and $\ga$ is topologically nilpotent. The general case can be reduced to this case by standard arguments. Then one observes that the fibers of $p$ are classical Springer fibers and deduce that its fiber over $X_\ga^{\mathrm{reg}}$ are $0$-dimensional and its fibers over the complement of $X_\ga^{\mathrm{reg}}$ are positive dimensional. From this one deduce that $\dim X_\ga^{\mathrm{reg}}=\dim Y_\ga$ and $\dim (X_\ga-X_\ga^{\mathrm{reg}})<\dim X_\ga^{\mathrm{reg}}$. These two equalities imply that $\dim X_\ga=\dim X_\ga^{\mathrm{reg}}$. 
\end{proof}

In fact, using global methods (the Hitchin fibration), Ng\^o is able to prove the following stronger result:
\begin{thm}\label{thm:ASF-density}
$X_\ga^{\mathrm{reg}}$ is dense in $X_\ga$ and $X_\ga$ is equi-dimensional.
\end{thm}
The proof uses Ng\^o's product formula \cite[4.15.1]{Ngo10} to reduce to showing the equi-dimensionality of Hitchin fibers, which follows from the flatness of the Hitchin fibration. Currently there is no purely local proof of the density and equi-dimensionality of $X_\ga$, although the equi-dimensionality of $Y_\ga$ is proved in \cite{KL} by purely local methods.  

\subsection{Comparing Kottwitz-Viehmann varieties and affine Springer fibers}
To prove the dimension formula for Kottwitz-Viehmann varieties $X_\ga^\la$ (Theorem~\ref{thm:dim-formula}), we try to generalize the methods for the affine Springer fibers $X_\ga$ outlined above. Using the theory of Vinberg monoid, one can define the regular locus $X_\ga^{\la,\mathrm{reg}}$ similarly as in the affine Springer fiber case. Also, one can define a suitable quotient $P_\ga$ of $LG_\ga$ through which its natural action on $X_\ga^\la$ factors. The first difference from the affine Springer fibers is that $X_\ga^{\la,\mathrm{reg}}$ is in general not a single $P_\ga$-torsor, but a union of finitely many $P_\ga$ torsors. This is not a big issue since we still have $\dim X_\ga^{\la,\mathrm{reg}}=\dim P_\ga$. Generalizing the methods in \cite{KL} and \cite{Be96} one can show that $\dim X_\ga^{\la,\mathrm{reg}}$ is given by the desired formula in the statement of Theorem~\ref{thm:dim-formula}.\par 
It remains to show that $\dim X_\ga^{\la,\mathrm{reg}}=\dim X_\ga^\la$. It is this step that manifests the major differences between the geometry of Kottwitz-Viehmann varieties and affine Springer fibers. The arguments for the analogous Theorem~\ref{thm:ASF-dim-equal} for affine Springer fibers fail to generalize to our situation. The main reason is that in general, there will be many irreducible components of $X_\ga^\la$ in the complement of $X_\ga^{\la,\mathrm{reg}}$, whereas for affine Springer fibers we know from Theorem~\ref{thm:ASF-density} that there are no components away from the regular locus. Thus the difficulty for proving Theorem~\ref{thm:dim-formula} is to bound the dimension of the irregular locus of $X_\ga^\la$. We are able to achieve this bound using global methods and currently there is no purely local proof yet.\par 
The difference also leads us to the interesting question of determining the number of irreducible components of $X_\ga^\la$, or more correctly, of the quotient $[X_\ga^\la/LG_\ga]$. For affine Springer fibers, we know from Proposition~\ref{prop:ASF-reg-torsor} and Theorem~\ref{thm:ASF-density} that the quotient $[X_\ga/LG_\ga]$ has only one irreducible component. However, for $X_\ga^\la$, we formulated a conjectural description (and proved in the case $\ga$ is hyperbolic) for the number of components of $[X_\ga^\la/LG_\ga]$ in terms of certain weight multiplicities in \S\ref{sec:irr-components}.

\section{Vinberg monoids and Kottwitz-Viehmann varieties}\label{sec:Vinberg}
In this section, we review the theory of Vinberg monoids and its role in the study of Kottwitz-Viehmann varieties. The main references are \cite{Vin95} and \cite{Ngo14}. See also \cite[\S2]{Chi}. \par 
For simplicity of exposition, we assume from now on that $G$ is semisimple simply-connected of rank $r$.

\subsection{Motivation and construction}
One reason we use the theory of Vinberg monoid to study Kottwitz-Viehmann varieties is that it provides a convenient framework to make sense of and analyze  the reduction mod $\varpi$ of an element in the double coset $K\varpi^\la K$.

\subsubsection{The case $G=\mathrm{SL}_2$} 
We first explain the idea in the ${SL}_2$ case. The Vinberg monoid for $G=\mathrm{SL}_2$
is $\ving=\mathrm{Mat}_2$, the multiplicative monoid of $2\times 2$ matrices. The basic idea is that via the determinant map $\det:\ving=\mathrm{Mat}_2\to\bbA^1$, the monoid $\ving$ realizes a flat degeneration of the group $G=\mathrm{SL}_2$, which is the fiber of $\det$ over $1$, to the semigroup of $2\times2$ matrices of determinant $0$.\par
Recall that the Cartan decomposition for $G=\mathrm{SL}_2$ reads
\[G(F)=\bigsqcup_{n\ge0}K\bm{\varpi^n & 0\\ 0 & \varpi^{-n}}K\]
For each $n\ge0$, we consider the map
\begin{equation}\label{eq:Cartan-to-Vinberg-SL2}
    \bigsqcup_{0\le i\le n}K\bm{\varpi^i & 0\\ 0 & \varpi^{-i}}K\to\mathrm{Mat}_2(\cO)
\end{equation}
defined by sending a matrix $g\in G(F)$ to $\varpi^ng$. In this way, we obtain a bijection from the left hand side above to the set of $2\times2$ matrices in $\mathrm{Mat}_2(\cO)$ with entries in $\cO$ whose determinant equals to $\varpi^{2n}$. Moreover, the double coset $K\bm{\varpi^n & 0\\ 0 & \varpi^{-n}}K$ is mapped bijectively onto the set of matrices with determinant $\varpi^{2n}$ in $\mathrm{Mat}_2(\cO)$ whose reduction mod $\varpi$ is nonzero (i.e., has at least one nonzero entry). Hence, via the map \eqref{eq:Cartan-to-Vinberg-SL2}, one can talk about the reduction mod $\varpi$ of elements in the double coset $K\bm{\varpi^n & 0\\ 0 & \varpi^{-n}}K$. The reduction does not land in the group $G$ (if $n>0$), but rather in the degenerate fiber $\det^{-1}(0)$ consisting of $2\times 2$ matrices of determinant $0$. 

\subsubsection{Construction for general semisimple simply-connected groups}
In general, the Vinberg monoid $\ving$ of $G$ is a certain affine algebraic monoid whose unit group is the central extension $G_+:=(T\times G)/Z_G$ of $G$ by the maximal torus $T$. The action of $G_+\times G_+$ on $G_+$ by left and right multiplication extends to $\ving$. Let $r=\dim T$ be the rank of $G$. Then there is a flat surjective morphism $\alpha:\ving\to\bbA^r$ extending the natural abelianization map $G_+\to T_{ad}:=T/Z_G$. Here we embed $T_{ad}$ into $\bbA^r$ by the $r$ simple roots. The fibers of $\alpha$ over $T_{ad}$ are all isomorphic to $G$ as $G\times G$-variety while the fibers over any ``coordinate axis" are viewed as various degenerations of $G$. The most degenerate fiber $\alpha^{-1}(0)$ is the so-called \emph{asymptotic semigroup of $G$}. The central $T$-action on $G_+$ extends to $\ving$ and the abelianization map $\alpha$ is $T$-equivariant where $T$ acts on $\bbA^r$ through the $r$ simple roots. There is a smooth open subscheme $\ving^0\subset\ving$ on which the center $T$ acts freely. The quotient $\ving^0/T$ is isomorphic to the wonderful compactification of the adjoint group $G_{ad}=G/Z_G$.\par 
In the case $G=\mathrm{SL}_2$ recalled above, we have $G_+=\mathrm{GL}_2$, $\ving=\mathrm{Mat}_2$, $\ving^0=\mathrm{Mat}_2-\{0\}$ and the central $T=\bbG_m$ acts by scaling. In general, we construct $\ving$ as follows. Let $\omega_i\in X^*(T)$ ($1\le i\le r$) be the fundamental weights. For each $1\le i\le r$, we extend the fundamental representation $\rho_i:G\to\mathrm{End}(V_{\omega_i})$ of heighest weight $\omega_i$ to a representation $\rho_i^+:G_+\to\mathrm{GL}(V_{\omega_i})$ by $\rho_i^+(t,g)=\omega_i(w_0(t)^{-1})\rho_i(g)$ where $w_0$ is the long element of the  Weyl group $W$. Also we extend the simple root $\alpha_i$ to a one-dimensional representation $\alpha_i^+:G_+\to\bbG_m^1$ by $\alpha_i^+(t,g)=\alpha_i(t)$. These maps define a homomorphism
\[(\alpha^+,\rho^+): G_+\to\bbG_m^r\times\prod_{i=1}^r \mathrm{GL}(V_{\omega_i}).\]
We define $\ving$ to be the normalization of the closure of $G_+$ in $\bbA^r\times\prod_{i=1}^r\mathrm{End}(V_{\omega_i})$ and define $\ving^0$ to be the inverse image of $\bbA^r\times\prod_{i=1}^r(\mathrm{End}(V_{\omega_i})-\{0\})$.

\subsection{Relation with Cartan decomposition}
For each dominant coweight $\la\in X_*(T)_+$, we define the $\cO$-scheme $V_\la$ (resp. $V_\la^0$) to be the pull back of $\alpha:\ving\to\bbA^r$ along the morphism $\spec\cO\to\bbA^r$ corresponding to the elements $(\varpi^{\langle-w_0(\la),\alpha_i\rangle})_{1\le i\le r}$. For any $\cO$-scheme $Y$, we let $L^+Y$ be the relative arc space classifying sections of the structure morphism $Y\to\spec\cO$. \par 
For each $\ga\in G(F)$ and $\la\in X_*(T)_+$, let $\ga_\la:=(\varpi^{-w_0(\la)},\ga)\in G_+(F)$. Since $\alpha(\ga_\la)=(\varpi^{\langle-w_0(\la),\alpha_i\rangle})_{1\le i\le r}$, we have $\ga_\la\in G_+(F)\cap V_\la(F)$.

\begin{prop}\label{prop:vin-cartan-decomposition}
The map $\ga\mapsto\ga_\la$ defines bijections 
\[K\varpi^\la K\xrightarrow{\sim}L^+V_\la^0(k),\]
\[\bigsqcup_{\substack{\mu\le\la \\ \mu\in X_*(T)_+}}K\varpi^\mu K\xrightarrow{\sim} L^+V_\la(k).\]
\end{prop}

As a consequence, we can rewrite the definition of $X_\ga^\la$ in a way that looks more similar to affine Springer fibers:
\[X_\ga^\la = \{g\in G(F)/K | g^{-1} \ga_\la g\in V_\la^0(\cO)\}=\{g\in G(F)/K|g^{-1}\ga_\la g\in\ving^0(\cO)\}\]
The second equality is because we always have $\alpha(g^{-1}\ga_\la g)=\alpha(\ga_\la)=(\varpi^{\langle-w_0(\omega_i),\la\rangle})$.\par
For technical reasons, we also need to consider the larger space
\[X_\ga^{\le\la}=\{g\in G(F)/K | g^{-1} \ga_\la g\in V_\la(\cO)\}=\{g\in G(F)/K|g^{-1}\ga_\la g\in\ving(\cO)\}\]
Here we only describe $X_\ga^\la$ and $X_\ga^{\le\la}$ set-theoretically. Later we will give their functor of points definition. 

\subsection{Adjoint quotient and regular centralizer of Vinberg monoids}

Let $\fC_+:=\ving//\mathrm{ad(G)}$ be the adjoint invariant quotient. It is shown in \cite{Bou15} that $\fC_+\cong\bbA^{2r}$ and the natural morphism $\chi_+:\ving\to\fC_+$ is surjective and flat. Each fiber of $\chi_+$ contains a unique closed adjoint $G$-orbit (the semisimple conjugacy class). However, unlike the Lie algebra case, in general the fibers of $\chi_+$ are reducible and contains more than one open orbits (regular conjugacy class). We let $\ving^{\mathrm{reg}}$ be the union of the regular conjugacy classes. Then $\ving^{\mathrm{reg}}$ is an open subset of $\ving^0$. The restriction $\chi_+^{\mathrm{reg}}:\ving^{\mathrm{reg}}\to\fC_+$ is smooth. The universal centralizer over $\ving$, whose fiber over $\ga$ consists of its centralizer in $G$, descends along $\chi_+^{\mathrm{reg}}$ to a commutative smooth group scheme $J$ over $\fC_+$.\par  
In \cite{Bou15}, certain sections of $\chi_+^{\mathrm{reg}}$ are defined. We briefly recall the construction here. Let $S\subset W$ be the set of simple reflections in $W$ corresponding to the choice of simple roots. An element $w\in W$ is called an \emph{$S$-Coxeter element} if in its reduced expression, each simple reflection in $S$ occurs precisely once. Let $\mathrm{Cox}(W,S)$ be the set $S$-Coxeter elements. For each $w\in \mathrm{Cox}(W,S)$, one can construct a section $\epsilon_+^w:\fC_+\to\ving^{\mathrm{reg}}$ generalizing the construction of Steinberg section for semisimple simply-connected groups. 
Unlike the Lie algebra case, the morphism $\fC_+\times G\to\ving^{\mathrm{reg}}$ defined by $(a,g)\mapsto g\epsilon_+^w(a)g^{-1}$ is not surjective. Instead, it maps onto an open subset $\ving^w\subset\ving^{\mathrm{reg}}$ and induces an isomorphism $(\fC_+\times G)/J\cong\ving^w$. The open subsets $\ving^w$ form an open cover of $\ving^{\mathrm{reg}}$ and each quotient $[\ving^w/\mathrm{Ad}(G)]$ is a $BJ$-gerbe neutralized by the section $\epsilon_+^w$. Moreover, for different $w,w'\in\mathrm{Cox}(W,S)$, the open subsets $\ving^w$ and $\ving^{w'}$ are not included in one another. This can be seen by looking at the \emph{nilpotent cone} $\cN:=\chi_+^{-1}(0)$ and its open subsets $\cN^0=\cN\cap\ving^0$ and $\cN^{\mathrm{reg}}$. We have the following description of irreducible components of $\cN$, $\cN^0$ and $\cN^{\mathrm{reg}}$.
\begin{prop}\label{prop:nil-cone-irr-components}
There are bijections
\[\mathrm{Cox}(W,S)\xrightarrow{\sim}\mathrm{Irr}(\cN^{\mathrm{reg}})\xrightarrow{\sim}\mathrm{Irr}(\cN^0)\xrightarrow{\sim}\mathrm{Irr}(\cN)\]
induced by sending $w$ to the irreducible component containing $\epsilon_+^w(0)$.
\end{prop}
This is \cite[Proposition 2.2.9]{Chi}. \par 
Hence the irreducible components of $\cN^{\mathrm{reg}}$ are $\ving^w\cap\cN$ for $w\in\mathrm{Cox}(W,S)$. Each component of $\cN^\mathrm{reg}$ consists of a single conjugacy class with representative $\epsilon_+^w(0)$. In particular, this shows that the intersections of the open sets $\ving^w$ with $\cN$ are mutually disjoint.

\subsection{Applications to Kottwitz-Viehmann varieties}
Now we can define the regular locus of a Kottwitz-Viehmann variety $X_\ga^\la$ as the open subscheme whose set of closed points is
\[X_\ga^{\la,\mathrm{reg}}:=\{g\in G(F)/K | g^{-1} \ga_\la g\in V_\la^{\mathrm{reg}}(\cO)\}\]

The non-emptiness criterion Theorem~\ref{thm:nonempty} is a consequence of the following more general result.

\begin{prop}[\cite{KoV},\cite{Chi}]\label{prop:KV-nonempty}
Let $\ga\in G(F)^{\mathrm{rs}}$ and $\la\in X_*(T)_+$. Let $\ga_\la:=(\varpi^{-w_0(\la)},\ga)\in G_+(F)\cap V_\la(F)$ as in Proposition~\ref{prop:vin-cartan-decomposition}. The following are equivalent
\begin{enumerate}
    \item $X_\ga^\la$ is nonempty;
    \item $X_\ga^{\le\la}$ is nonempty;
    \item $X_\ga^{\la,\mathrm{reg}}$ is nonempty;
    \item $\kappa_G(\ga)=p_G(\la)$ and $\nu_\ga\le_\bbQ\la$;
    \item $\chi_+(\ga_\la)\in\fC_+(\cO)$.
\end{enumerate}
\end{prop}
The implication (1)$\Rightarrow$(4) is proved in \cite{KoV}. The remaining is in \cite[Proposition 3.1.6]{Chi}.\par 
Note that since we have assumed $G$ to be semisimple simply-connected, the condition $\kappa_G(\ga)=p_G(\la)$ is automatic. However the result is true for any split connected reductive group. \par 
From now on, we assume that $a:=\chi_+(\ga_\la)\in\fC_+(\cO)$. Let $\bar a\in\fC_+(k)$ be the reduction of $a$ mod $\varpi$. 

Let $J_a$ be the $\cO$-scheme defined as the pull-back of $J$ along $a:\spec\cO\to\fC_+$. Then $J_a$ is an integral model of $G_\ga$. We consider the commutative $k$-group $P_a:=G_\ga(F)/J_a(\cO)$ that acts on the varieties $X_\ga^\la$, $X_\ga^{\le\la}$ and $X_\ga^{\la,\mathrm{reg}}$. The action on the regular locus $X_\ga^{\la,\mathrm{reg}}$ is easier to analyze. See \cite[\S3.9.10]{Chi} for details. 

\begin{prop}[\cite{Chi}]
The morphism $X_\ga^{\la,\mathrm{reg}}\to [(\chi_+^{\mathrm{reg}})^{-1}(\bar a)/\mathrm{Ad}(G)]$ which sends $g$ to the reduction mod $\varpi$ of $g^{-1}\ga_\la g$ induces a bijection between the quotient $X_\ga^{\la,\mathrm{reg}}/P_a$ and the set of irreducible components of $(\chi_+^{\mathrm{reg}})^{-1}(\bar a)$.
\end{prop}

Since $\ving^{\mathrm{reg}}=\cup_{w\in\mathrm{Cox(W,S)}}\ving^w$ and each quotient $[\ving^w/\mathrm{Ad}(G)]$ is a $BJ$-gerbe, the number of irreducible components of the extended Steinberg fiber $(\chi_+^{\mathrm{reg}})^{-1}(\bar a)$ is bounded above by $\#\mathrm{Cox}(W,S)$. Moreover, we know from Proposition~\ref{prop:nil-cone-irr-components} that equality is reached for $(\chi_+^{\mathrm{reg}})^{-1}(0)=\cN^{\mathrm{reg}}$. This proves the following

\begin{prop}[\cite{Chi}]\label{prop:reg-irr-components}
We have $\#(X_\ga^{\la,\mathrm{reg}}/P_a)\le\#\mathrm{Cox}(W,S)$. If $a:=\chi_+(\ga_\la)\equiv0\mod\varpi$, then equality holds.
\end{prop}
This is \cite[Corollary 3.9.12]{Chi}
Comparing Proposition~\ref{prop:reg-irr-components} and Theorem~\ref{thm:irr-component-hyperbolic}, we get the following purely combinatorics result

\begin{cor}
Let $\la,\mu\in X_*(T)_+$ be dominant coweights. Suppose $\la$ lies in the interior of the Weyl chamber and $\la-\mu$ lies in the interior of the positive coroot cone, then we have $m_{\la\mu}\ge\#\mathrm{Cox}(W,S)$. 
\end{cor}
\begin{proof}
Choose an element $\ga\in\varpi^\mu T(\cO)\cap G(F)^{\mathrm{rs}}$. Let $a:=\chi_+(\ga_\la)$ as before. The regularity conditions on $\la$ and $\la-\mu$ gaurantees that $\bar a=0$. Hence by Proposition~\ref{prop:reg-irr-components}, we have $\#\mathrm{Irr}[X_\ga^{\la,\mathrm{reg}}/ G_\ga(F)]=\#\mathrm{Cox}(W,S)$. On the other hand, since $\ga$ is hyperbolic, we have $\nu_\ga=\mu$ and $\#\mathrm{Irr}[X_\ga^\la/ G_\ga(F)]=m_{\la\mu}$ by Theorem~\ref{thm:irr-component-hyperbolic}.
\end{proof}
It would be interesting to find a straightforward proof of this combinatorics result.\par 
Notice that $\#\mathrm{Cox}(W,S)$ is a constant independant of $\la,\mu$ while the weight multiplicity $m_{\la\mu}$ grows when $\la$ increases. Hence in general we expect most irreducible components of $X_\ga^\la$ to be irregular. The hard step for proving Theorem~\ref{thm:dim-formula} is to bound the dimension of the irregular locus and show that $\dim X_\ga^{\la,\mathrm{reg}}=\dim X_\ga^\la$. On any irregular component of $X_\ga^\la$, the group $P_a$ will have infinitely many orbits and it appears harder to bound their dimension. In the next section, we discuss the global method to prove this bound.

\section{The Hitchin-Frenkel-Ng\^o fibration}\label{sec:Hitchin}
\subsection{Generalities on arc spaces and mapping stacks}
For any $k$-stack $\cX$, recall that the arc stack $L^+\cX$ and loop stack $L\cX$ are groupoid-valued functors that associates any $k$-algebra the groupoinds
\[L^+\cX(R)=\cX(R[[\varpi]]),\quad L\cX(R)=\cX(R((\varpi))).\]
For any point $\ga\in\cX(F)$, let $(L^+\cX)_\ga$ be the fiber of the natural morphism $L^+\cX\to L\cX$ over $\ga\in (L\cX)(k)=\cX(F)$. Then for any $k$-algebra $R$ the groupoid $(L^+\cX)_\ga(R)$ classifies pairs $(\alpha,\iota)$ where $\alpha\in\cX(R[[\varpi]])$ and $\iota:\alpha|_{\spec R((\varpi))}\xrightarrow{\sim}\ga$ is an isomorphism in $L\cX(R)$. We note that if $\cX$ has separated diagonal, then objects in $(L^+\cX)_\ga(R)$ have no non-trivial automorphisms. Also, if $\ga_1,\ga_2\in\cX(F)$ are isomorphic, then any isomorphism between $\ga_1$ and $\ga_2$ induces an isomorphism $(L^+\cX)_{\ga_1}\cong (L^+\cX)_{\ga_2}$.
In particular, if we denote by $\mathrm{Aut}_\ga$ the automorphism $F$-group scheme of $\ga$, then there is a natural action of the loop group $L\mathrm{Aut}_\ga$ on $(L^+\cX)_\ga$.\par 
A standard method of studying $(L^+\cX)_\ga$ and the action of $L\mathrm{Aut}_\ga$ is to consider the global mapping stack of $\cX$ for a projective smooth curve $C$. To be more precise, let $\mathrm{Map}(C,\cX)$ be the stack over $k$ such that for any $k$-algebra $R$, its groupoid of $R$-points classify morphisms $C\otimes_kR\to\cX$.\par 
In practice, we will usually consider certain open substack of the mapping stack. For any open substack $\cX'\subset\cX$, the stack $\mathrm{Map}^\circ(C,\cX\supset\cX')$ classifies maps $C\to \cX$ such that the generic point of $C$ is mapped into the open substack $\cX'$

\subsection{Kottwitz-Viehmann varieties via arc spaces}
We apply the previous local construction to the stacks $[\ving/\mathrm{Ad}(G)]$ (resp. $[\ving^0/\mathrm{Ad}(G)]$, $[\ving^{\mathrm{reg}}/\mathrm{Ad}(G)]$) and their $F$-points defined by $\ga_\la\in G_+(F)^{\mathrm{rs}}\subset\ving^{\mathrm{reg}}(F)$, where $\ga_\la:=(\varpi^{-w_0(\la)},\ga)$ is as in Proposition~\ref{prop:vin-cartan-decomposition}, to obtain spaces
\[\cX_\ga^{\le\la}:=(L^+[\ving/\mathrm{Ad}(G)])_\ga,\quad
\cX_\ga^\la:=(L^+[\ving^0/\mathrm{Ad}(G)])_\ga,\quad 
\cX_\ga^{\la,\mathrm{reg}}:=(L^+[\ving^{\mathrm{reg}}/\mathrm{Ad}(G)])_\ga
\]

These spaces are highly non-reduced. Taking their underlying reduced structures we get the locally of finite type schemes $X_\ga^{\le\la}$, $X_\ga^\la$ and $X_\ga^{\la,\mathrm{reg}}$ whose set-theoretic description we have seen before.\par 
Unraveling the definitions, we have the following more concrete description of the spaces $\cX_\ga^{\le\la}$ (and similarly for $\cX_\ga^\la$ and $\cX_\ga^{\la,\mathrm{reg}}$). For any $k$-algebra $R$, the set $\cX_\ga^{\le\la}(R)$ consists of isomorphism classes of triples $(E,\varphi,\iota)$ where
\begin{itemize}
    \item $E$ is a $G$-torsor on $\spec R[[\varpi]]$;
    \item $\varphi$ is a section of the scheme $E\times^G\ving$ over $\spec R[[\varpi]]$ (which is equivalent to a $G$-equivariant morphism $E\to\ving$);
    \item $\iota$ is a trivialization of $E|_{\spec R((\varpi))}$ such that under the induced isomorphism 
    \[(E|_{\spec R((\varpi))})\times^G\ving\cong\ving\times\spec R((\varpi))\]
    the section $\varphi$ is taken to (the graph of) $\ga_\la$.
\end{itemize}

We assume that $a:=\chi_+(\ga_\la)$ lies in $\fC_+(\cO)$ so that $X_\ga^\la$ is nonempty by Proposition~\ref{prop:KV-nonempty}. Then for any $w\in\mathrm{Cox}(W,S)$, $\ga$ and $\epsilon_+^w(a)$ are isomorphic in the groupoid $[\ving/\mathrm{Ad}(G)](F)$ (resp. $[\ving^0/\mathrm{Ad}(G)](F)$ and $[\ving^{\mathrm{reg}}/\mathrm{Ad}(G)](F)$). Hence the scheme $X_\ga^{\le\la}$ (resp. $X_\ga^\la$ and $X_\ga^{\la,\mathrm{reg}}$) is isomorphic to $X_{\epsilon_+^w(a)}^{\le\la}$ (resp. $X_{\epsilon_+^w(a)}^\la$ and $X_{\epsilon_+^w(a)}^{\la,\mathrm{reg}}$). In other words, the isomorphism classes of $X_\ga^{\le\la}$, $X_\ga^\la$ and $X_\ga^{\la,\mathrm{reg}}$ only depend on $a=\chi_+(\ga_\la)$. For any $a\in\fC_+(\cO)\cap\fC_+(F)^{\mathrm{rs}}$, we simply denote $X_a:=X_{\epsilon_+^w(a)}^{\le\la}$ and $X_a^{\mathrm{reg}}:=X_{\epsilon_+^w(a)}^{\la,\mathrm{reg}}$ for some fixed $w\in\mathrm{Cox}(W,S)$. 

\subsection{Hitchin-Frenkel-Ng\^o fibration}
Now we apply the global mapping stack construction. See \cite[\S4]{Chi} for details. It turns out that to get big enough mapping stacks, one has to quotient by the central torus and consider the stack 
\[\cM:=\mathrm{Map}^\circ(C,[\ving/Z_+\times\mathrm{Ad}(G)]\supset [G_+^{\mathrm{rs}}/Z_+\times\mathrm{Ad}(G)])\] 
where $Z_+\cong T$ is the center of $G_+$. 
Let $\fC_+^\circ:=G_+^{\mathrm{rs}}//\mathrm{Ad}(G)$ be the open subscheme of $\fC_+$ consisting of regular semisimple conjugacy classes in the unit group $G_+$. In particular, the image of $\fC_+^\circ$ under the projection to the universal abelianization $\fC_+\to\bbA^r$ equals $\bbG_m^r$, the complement of the coordinate axis in $\bbA^r$. Also consider the stack \[\cA:=\mathrm{Map}^\circ(C,[\fC_+/Z_+]\supset [\fC_+^0/Z_+]).\]

There are natural morphisms from $\cM$ and $\cA$ to $\mathrm{Bun}_{Z_+}=\mathrm{Map}(C,BZ_+)$, the moduli stack of $Z_+$-torsors on $C$. Moreover, there is a natural morphism $f:\cM\to\cA$ over $\mathrm{Bun}_{Z_+}$. In practice, we usually fix a $Z_+$-torsor $\cL$ on $C$ and consider the fibers $\cM_\cL$ (resp. $\cA_\cL$) of $\cM$ (resp. $\cA$) over $\cL$, together with the morphism 
$f_\cL:\cM_\cL\to\cA_\cL$. For suitably chosen $\cL$, the stack $\cA_\cL$ is represented by an open subscheme of a $k$-vector space and any extended Steinberg section $\epsilon_+^w:\fC_+\to\ving^{\mathrm{reg}}$ induces a section of $f_\cL$. We assume this holds from now on and fix a section $\epsilon_\cL:\cA_\cL\to\cM_\cL$ of $f_\cL$ that is induced by $\epsilon_+^w$ for some $w\in\mathrm{Cox}(W,S)$. We refer to the morphism $f_\cL$ as the \emph{Hitchin-Frenkel-Ng\^o fibration}. \par 
More concretely, the stack $\cM_\cL$ classifies pairs $(\cE,\varphi)$ where $\cE$ is a $G$-torsor on $C$ and $\varphi$ is a section of the scheme $\cE\times^G\ving^\cL$ over $C$, where $\ving^\cL:=\cL\times^{Z_+}\ving$. We refer to such a pair as a \emph{Hitchin-Vinberg pair}. The morphism $f_\cL$ sends the pair $(\cE,\varphi)$ to $\chi_+(\varphi)$, which defines a section of $\fC_+^\cL:=\cL\times^{Z_+}\fC_+$ over $C$.

Replacing $\ving$ by its open subschemes $\ving^0$ (resp. $\ving^{\mathrm{reg}}$) in the definition of $\cM_\cL$, we define open substacks $\cM_\cL^0\supset\cM_\cL^{\mathrm{reg}}$. The image of the global section $\epsilon_\cL$ lies in $\cM_\cL^{\mathrm{reg}}$.

\subsubsection{Relation with Kottwitz-Viehmann varieties}
The regular centralizer $J$ over $\fC_+$ is $Z_+$-equivariant and descends to a commutative smooth group scheme over $[\fC_+/Z_+]$, which we still denote by $J$. For any closed point $a\in\cA_\cL$, let $J_a$ be the commutative smooth group scheme over $C$ defined as the pull-back of $J$ along $a:C\to[\fC_+/Z_+]$. Let $\cP_a:=\mathrm{Bun}_{J_a}=\mathrm{Map}(C,BJ_a)$ be the Picard stack of $J_a$-torsors on $C$. The $Z_+$-equivariant action of $BJ$ on the stack $[\ving/\mathrm{Ad}(G)]$ (and its open substacks $[\ving^0/\mathrm{Ad}(G)]$, $[\ving^{\mathrm{reg}}/\mathrm{Ad}(G)]$) induces an action of the Picard stack $\cP_a$ on the fiber $\cM_a:=f_\cL^{-1}(a)$ (and its open substacks $\cM_a^0=\cM_a\cap\cM_\cL^0$, $\cM_a^{\mathrm{reg}}:=\cM_a\cap\cM_\cL^{\mathrm{reg}}$). This is the global analogue of the local $P_a$ action on the Kottwitz-Viehmann varieties.\par 
There is an open subset $\cA_\cL^{\mathrm{ani}}\subset\cA_\cL$ consisting of $a\in\cA_\cL$ such that $\cP_a$ is of finite type. With suitable choice of $\cL$, this subset will be non-empty.
\begin{thm}[Product formula]\label{thm:prod-formula}
For each closed point $a\in\cA_\cL^{\mathrm{ani}}$, let $U_a$ be the nonempty open subset of $C$ defined as the inverse image of $[\fC_+^{\mathrm{rs}}/Z_+]$ under $a:C\to[\fC_+/Z_+]$. Then there is a homeomorphism of stacks
\[\prod_{v\in C-U_a}[X_{a_v}/P_{a_v}]\cong [\cM_a/\cP_a]\]
\end{thm}
This is analogous to Ng\^o's product formula in \cite[\S4.15]{Ngo10}. 
The idea is as follows. For each $v\in C-U_a$, the local data $x_v\in X_{a_v}$ defines a Hitchin-Vinberg pair over the formal disc $\spec\cO_v$. We glue them with the restriction of $\epsilon_\cL(a)$ over $U_a$ to obtain a Hitchin-Vinberg pair on $C$. Conversely, for any Hitchin-Vinberg pair $(\cE,\varphi)$ in $\cM_a$, after twisting by an element in the global Picard $\cP_a$, we may assume that the restriction of $(\cE,\varphi)$ to $U_a$ is isomorphic to $\epsilon_\cL(a)$. Fixing such an isomorphism, then the restriction of $(\cE,\varphi)$ to the punctured discs $\spec\cO_v$ for $v\in C-U_a$ gives points in the Kottwitz-Viehmann varieties $X_{a_v}$.\par

Using Theorem~\ref{thm:prod-formula}, and some delicate local-global approximation arguments, the key step in the proof of Theorem~\ref{thm:dim-formula} (i.e. showing that $\dim X_\ga^\la=\dim P_a$) is reduced to proving the global statement that $\dim\cM_a=\dim\cP_a$. For this, we use a specialization argument to reduce to the dimension formula of Kottwitz-Viehmann varieties $X_\ga^\la$ such that either $\ga$ is hyperbolic conjugacy class, or $\la=0$. Both cases can be proven by purely local methods.\par 
There are two key ingredients behind this specialization argument: first, as $a$ varies, the Picard stacks $\cP_a$ form a \emph{smooth} family over $\cA_\cL$; second, the fibration $f_\cL$ is proper when restricted to $\cA_\cL^{\mathrm{ani}}$ so one can use semi-continuity of fiber dimension and the specialization argument to bound the dimension $\cM_a$ by the dimension of $\cP_a$. Here we need semi-continuity on the target of the morphism, for which properness is crucial.

\section{An example in $\mathrm{SL}_3$ case}\label{sec:example}
The goal of this section is to illustrate some differences between geometry of affine Springer fibres in the Lie algebra case and in the group case (i.e. the Kottwitz-Viehmann varieties) by studying a concrete example for the group $\mathrm{SL}_3$.\par 
Recall that in the Lie algebra case, the affine Springer fibers has a dense open subset (the ``regular locus") which is a torsor under a commutative algebraic group. Hence questions concerning geometries of such affine Springer fibers (for example dimension, irreducible components, equidimensionality etc.) could be answered by studying this commutative algebraic group, whose geometry is simpler than the affine Springer fiber itself.\par 
In the group case, one can still define the notion of ``regular locus" and construct an action of a commutative algebraic group. However, the example we present below will show the following:
\begin{itemize}
\item The ``regular locus" is no longer a torsor under this algebraic group in general (i.e. the action is not transitive).
\item The ``regular locus" is not dense in general and there will be irreducible components disjoint from the "regular locus".
\end{itemize} 
Our example will have the following feature: it has 3 irreducible components (after quotient by the action of the commutative algebraic group), among which 2 components are generically regular and the remaining one is irregular. We are also able to describe the generically regular components using the theory of Lusztig-He. \par 
Throughout this section we assume for simplicity that $k$ is an algebraically closed field of characteristic zero.

\subsection{Vinberg monoid for SL(3)}
\subsubsection{}
Consider the vector space $V=k^3$ with standard basis $v_1,v_2,v_3$. Then $\wedge^2 V\cong k^3$ where the isomorphism is given by the basis
\[\{v_2\wedge v_3, v_3\wedge v_1, v_1\wedge v_2\}\]
For
\[g=\begin{bmatrix}
a_{11}&a_{12}&a_{13}\\a_{21}&a_{22}&a_{23}\\a_{31}&a_{32}&a_{33}
\end{bmatrix}\in\mathrm{GL}_3,\]
under the above basis, the matrix of $g$ acting on $\wedge^2V$
is
\[\wedge^2g=\begin{bmatrix}
a_{22}a_{33}-a_{23}a_{32} & a_{23}a_{31}-a_{21}a_{33} & a_{21}a_{32}-a_{31}a_{22}\\
a_{13}a_{32}-a_{12}a_{33} & a_{11}a_{33}-a_{13}a_{31} & a_{12}a_{31}-a_{11}a_{32}\\
a_{12}a_{23}-a_{13}a_{22} & a_{13}a_{21}-a_{11}a_{23} & a_{11}a_{22}-a_{12}a_{21}
\end{bmatrix}=(\det g)\cdot {}^tg^{-1}\]

\subsubsection{}
Let $G=\mathrm{SL}_3$, $T$ the diagonal torus in $G$ and $U$ the group of upper triangular unipotent matrices. Consider the enhanced group $G_+:=(T\times G)/Z_G$ where $Z_G\cong\mu_3$ acts anti-diagonally. Consider the open embedding
\begin{equation}\label{embed-G+-equation}
\xymatrix@R=1pt{
G_+\ar[r] & \bbA^2\times\mathrm{Mat}_3\times\mathrm{Mat}_3\\
(t,g)\ar@{|->}[r] & (\alpha_1(t),\alpha_2(t),\omega_1(t)g,\omega_2(t)({}^tg^{-1})) 
}
\end{equation}
The Vinberg monoid $V_G$ is defined as the normalization of the closure of $G_+$ in $\bbA^2\times\mathrm{Mat}_3\times\mathrm{Mat}_3$. There is a smooth open subscheme $V_G^0\subset V_G$ defined as the normalization of the closure of $G_+$ in $\bbA^2\times(\mathrm{Mat}_3-\{0\})\times(\mathrm{Mat}_3-\{0\})$.\par 
The group $G$ acts by conjugation on $V_G$. For any $\gamma\in V_G$, let $G_\gamma$ be the centralizer of $\gamma$ in $G$ under this adjoint action. We define the regular locus of the Vinberg monoid to be
\[V_G^{\mathrm{reg}}=\{\gamma\in V_G|\dim(G_\gamma)=2\}\]
where we recall that $G=\mathrm{SL}_3$ has rank 2.\par 
The centralizer can be calculated explicily using the embedding \eqref{embed-G+-equation}. In other words, if an element $\gamma\in V$ is represented by the quadruple
\[\gamma=(a_1,a_2,A,B)\in\bbA^2\times\mathrm{Mat}_3\times\mathrm{Mat}_3\]
then for any $h\in G$, we have
\[\mathrm{Ad}(h)\gamma=(a_1,a_2,hAh^{-1}, {}^th^{-1}B({}^th))\]
This allow us to explicitly determine the centralizer $G_\gamma$.

\subsection{Conjugacy classes in nilpotent cone}
Recall the extended Steinberg map
\[\xymatrix@R=1pt{
\chi_+: V\ar[r] & \fC_+=\bbA^2\times\bbA^2\\
(a_1,a_2,A,B)\ar@{|->}[r] & (a_1,a_2,\mathrm{tr}(A),\mathrm{tr}(B))
}\]
The nilpotent cone is by definition $\mathcal{N}:=\chi_+^{-1}(0)$. Moreover, we let $\mathcal{N}^0=\mathcal{N}\cap V^0$ and $\mathcal{N}^{\mathrm{reg}}:=\mathcal{N}\cap V^{\mathrm{reg}}$.\par 
In this section, we recall the adjoint $G$-orbits on $\mathcal{N}$ following Lusztig-He. In particular, we will see that the conjugacy classes in $\mathcal{N}^{\mathrm{reg}}$ corresponds bijectively to Coxeter elements in the Weyl group. Since for $G=\mathrm{SL}_3$ there are 2 Coxeter elements (once we fix the simple reflections in the Weyl group), hence there are 2 conjugacy classes in $\mathcal{N}^{\mathrm{reg}}$.\par 
The Weyl group for $G=\mathrm{SL}_3$ is $W=S_3$, which is generated by two simple reflections $s_1=(12)$, $s_2=(23)$. We fix representatives of the simple reflections in $G$ as follows:
\[\dot{s_1}=\begin{bmatrix}
0&-1&0\\1&0&0\\0&0&1
\end{bmatrix},\quad\dot{s_2}=\begin{bmatrix}
1&0&0\\0&0&-1\\0&1&0
\end{bmatrix}\]
The long element of the $W$ will be $w_0=(13)$ with representative $\dot{w_0}=\begin{bmatrix}
0&0&1\\0&-1&0\\1&0&0
\end{bmatrix}$. \par 
The following idempotents in $V$ will play an important role in the description of conjugacy classes:
\[e_0=\left(0,0,\begin{bmatrix}1&0&0\\0&0&0\\0&0&0\end{bmatrix},\begin{bmatrix}0&0&0\\0&0&0\\0&0&1\end{bmatrix}\right),\quad h_0=w_0e_0w_0=\left(0,0,\begin{bmatrix}0&0&0\\0&0&0\\0&0&1\end{bmatrix},\begin{bmatrix}1&0&0\\0&0&0\\0&0&0\end{bmatrix}\right)\]

\[e_1=\left(1,0,\begin{bmatrix}1&0&0\\0&1&0\\0&0&0\end{bmatrix},\begin{bmatrix}0&0&0\\0&0&0\\0&0&1\end{bmatrix}\right),\quad h_1=w_0e_1w_0=\left(1,0,\begin{bmatrix}0&0&0\\0&1&0\\0&0&1\end{bmatrix},\begin{bmatrix}1&0&0\\0&0&0\\0&0&0\end{bmatrix}\right)\]

\[e_2=\left(0,1,\begin{bmatrix}1&0&0\\0&0&0\\0&0&0\end{bmatrix},\begin{bmatrix}0&0&0\\0&1&0\\0&0&1\end{bmatrix}\right),\quad h_2=w_0e_2w_0=\left(0,1,\begin{bmatrix}0&0&0\\0&0&0\\0&0&1\end{bmatrix},\begin{bmatrix}1&0&0\\0&1&0\\0&0&0\end{bmatrix}\right)\]

In the notation of Vinberg and Lusztig-He, we have $e_0=e_{\varnothing}$, $e_1=e_{\{\alpha_1\}}$, $e_2=e_{\{\alpha_2\}}$ and similarly for $h_i$. \par 
There are 2 Coxeter elements in the Weyl group $W$:  $w_1=s_1s_2=(123)$ and $w_2=s_2s_1=(132)$. We choose representatives of $w_1,w_2$ in $G$ as follows:
\[\dot{w_1}=\begin{bmatrix}0&0&1\\1&0&0\\0&1&0\end{bmatrix},\quad\dot{w_2}=\begin{bmatrix}
0&1&0\\0&0&1\\1&0&0
\end{bmatrix}\]
According to Lusztig-He, the following will be representatives of the 2 conjugacy classes in $\mathcal{N}^{\mathrm{reg}}$:
\[\dot{w_1}h_0=\left(0,0,\begin{bmatrix}0&0&1\\0&0&0\\0&0&0\end{bmatrix},\begin{bmatrix}
0&0&0\\1&0&0\\0&0&0\end{bmatrix}\right),\quad
\dot{w_2}h_0=\left(0,0,\begin{bmatrix}
0&0&0\\0&0&1\\0&0&0\end{bmatrix},\begin{bmatrix}
0&0&0\\0&0&0\\1&0&0\end{bmatrix}\right)\]
Let $\mathcal{N}_i^{\mathrm{reg}}:=\mathrm{Ad}(G)(\dot{w_i}h_0)$ ($i=1,2$) be the corresponding orbits. Then we have
\[\mathcal{N}^{\mathrm{reg}}=\cN_1^{\mathrm{reg}}\sqcup\cN_2^{\mathrm{reg}}.\]
Let $\cN^0_{\mathrm{irreg}}$ be the union of non-regular orbits in $\cN^0$, then we have
\[\cN^0=\cN_1^{\mathrm{reg}}\sqcup\cN_2^{\mathrm{reg}}\sqcup\cN^0_{\mathrm{irreg}}\]
Then according to Lusztig-He we have 
\[\cN^0_{\mathrm{irreg}}=\mathrm{Ad}(G)[(B\times B)\cdot\dot{w_0}h_0]\]
We have
\[\dot{w_0}h_0=\left(0,0,\begin{bmatrix}
0&0&1\\0&0&0\\0&0&0
\end{bmatrix}, \begin{bmatrix}
0&0&0\\0&0&0\\1&0&0
\end{bmatrix}\right)\]
and
\[(B\times B)\cdot\dot{w_0}h_0=\{n(a,b):=\left(0,0,\begin{bmatrix}
0&0&a\\0&0&0\\0&0&0
\end{bmatrix}, \begin{bmatrix}
0&0&0\\0&0&0\\b&0&0
\end{bmatrix}\right)\in V| a,b\in k^\times\}\]
It is easy to see that $n(a,b)$ is conjugate to $n(c,d)$ if and only if $a/b=c/d$. Hence the conjugacy classes in $\cN^0_{\mathrm{irreg}}$ corresponds bijectively to $k^\times$, and $n(\alpha,1)$ is a representative of the conjugacy classes corresponding to $\alpha\in k^\times$. In particular, there are infinitly many conjugacy classes in $\cN^0_{\mathrm{irreg}}$.\par
The adjoint action of $B$ on $(B\times B)\cdot w_0 h_0$ is described concretely as follows: an element in $B$ whose diagonal entries are $\la_1,\la_2,\la_3$ sends $n(a,b)$ to $n(\la_1\la_2^{-1}a,\la_1\la_2^{-1}b)$.
\subsubsection{}
Recall that $V^0_G$ is a $Z_+=\bbG_m^2$-torsor over the wonderful compactification $\overline{G}$ of $G_{\mathrm{ad}}=PGL_3$. Let $\overline{\cN}$ (resp. $\overline{\cN}^{\mathrm{reg}}_1$, $\overline{\cN}^{\mathrm{reg}}_1$, $\overline{\cN}^{\mathrm{irreg}}$) be the image of the corresponding objects in $\overline{G}$. Then the adjoint action of $G$ has 3 orbits on $\overline{\cN}$: the orbits $\overline{\cN}^{\mathrm{reg}}_1$ and  $\overline{\cN}^{\mathrm{reg}}_2$ are open in $\overline{\cN}$, while the orbit 
$\overline{\cN}^{\mathrm{irreg}}$ is a single point, hence closed.\par 
In other words, the infinity of adjoint orbits in $\cN$ is caused by the torus $Z_+$.\par

\subsection{The example of a Kottwitz-Viehmann variety}
Let $\lambda=(2,0,-2)\in X_*(T)_+$. Then $-w_0(\lambda)=\lambda$ and 
\[\varpi^{-w_0(\lambda)}=\varpi^\lambda=\begin{bmatrix}\varpi^2&0&0\\0&1&0\\0&0&\varpi^{-2}\end{bmatrix}\in T(F)\]
Let $\gamma=\left[\begin{smallmatrix}2&0&0\\0&1&0\\0&0&\frac{1}{2}\end{smallmatrix}\right]\in T(F)$ and consider the affine Springer fiber
\[X_\gamma^{\le\lambda}=\{g\in G(F)/K | g^{-1}\gamma g\in\overline{K\varpi^\lambda K}\}.\]
We can write (set theoretically) $X_\gamma^{\le\lambda}=Y_\gamma^{\le\lambda}\times X_*(T)$ where
\[Y_\gamma^{\le\lambda}=\{u\in U(F)/U(\cO) | u^{-1}\gamma u\in\overline{K\varpi^\lambda K}\}\]
\begin{lem}
$Y_\gamma^{\le\lambda}=S_0\cap\mathrm{Gr}_{\le\lambda}$ where $S_0=U(F)/U(\cO)$ is a semiinfinite orbit in the affine Grassmanian. In other words, $Y_\gamma^{\lambda}$ is the MV-cycle associated to the coweights $\mu=0$ and $\lambda=(2,0,-2)$
\end{lem}
\begin{proof}
Write $u=\left[\begin{smallmatrix}1&x&z\\0&1&y\\0&0&1\end{smallmatrix}\right]$, then 
\[u^{-1}=\begin{bmatrix} 1&-x&xy-z\\0&1&-y\\0&0&1 \end{bmatrix},\quad
u^{-1}\gamma u=\begin{bmatrix} 2&x&\frac{3z-xy}{2}\\0&1&\frac{1}{2}y\\0&0&1 \end{bmatrix},\quad
{}^t(u^{-1}\gamma u)^{-1}=\begin{bmatrix}\frac{1}{2}&0&0\\-\frac{1}{2}x &1&0\\
\frac{2xy-3z}{2} & -y & 2 \end{bmatrix}\]
Then $u\in Y_\gamma^{\lambda}$ means that the pole order of entries in the matrices $u^{-1}\gamma u$ and ${}^t(u^{-1}\gamma u)^{-1}$
are bounded by $2$. From the formula above, this is the same as saying that $u\in\overline{K\varpi^{\lambda}K}$.
\end{proof}

\subsection{Irreducible components}
In the irreducible representation of $\mathrm{PGL}_3(\bbC)$ with highest weight $\lambda=(2,0,-2)$, the weight space of $\mu=(0,0,0)$ has dimension $3$. By geometric Satake equivalence, this means that $Y_\gamma^{\le\lambda}=S_0\cap\mathrm{Gr}_{\le\lambda}$ has 3 irreducible components. Consider the following 3 subscheme of $Y_\gamma^{\le\lambda}$:
\[Y_1=\{u=\left[\begin{smallmatrix}1&x&z\\0&1&y\\0&0&1\end{smallmatrix}\right]\in U(F)| \mathrm{ord}(x)=-2,\mathrm{ord}(y)=0,\mathrm{ord}(z)\ge-2\}/U(\cO)\] 
\[Y_2=\{u=\left[\begin{smallmatrix}1&x&z\\0&1&y\\0&0&1\end{smallmatrix}\right]\in U(F)| \mathrm{ord}(x)=0,\mathrm{ord}(y)=-2,\mathrm{ord}(z)\ge-2\}/U(\cO)\] 
\[Y_3=\{u=\left[\begin{smallmatrix}1&x&z\\0&1&y\\0&0&1\end{smallmatrix}\right]\in U(F)| \mathrm{ord}(x)=-1,\mathrm{ord}(y)=-1,\mathrm{ord}(z)\ge-2\}/U(\cO)\] 
Then the closures of $Y_i (i=1,2,3)$ in $Y_\gamma^{\le\lambda}$ are its irreducible components. It is easy to see from this explicit description that $Y_\gamma^{\le\lambda}$ is equidimensional of dimension 4. This also follows from general theory of geometric Satake, since $\langle 2\rho,\lambda+\mu\rangle=4$ in our case. 

\subsection{The regular locus}
In this section, we determine the regular locus of the example of affine Springer $X_\gamma^{\le\lambda}$ studied in the previous part. It will be sufficient to determine the regular locus in $Y_\gamma^{\le\lambda}$. First recall the definition of this regular locus.
\begin{defn}
The regular locus $Y_\gamma^{\mathrm{reg}}\subset Y_\gamma^{\le\lambda}$ is defined by
\[Y^{\mathrm{reg}}_\gamma=\{u\in Y_\gamma^{\le\lambda} | (\varpi^{-w_0(\lambda)},u^{-1}\gamma u)\in G_+(F)\cap V^{\mathrm{reg}}(\cO)\}\]
\end{defn}
\begin{prop}
The generic open subset of the components $Y_1$ and $Y_2$ are contained in $Y_\gamma^{\mathrm{reg}}$, while the generic open subset of $Y_3$ is disjoint from $Y_\gamma^{\mathrm{reg}}$.
\end{prop}
\begin{proof}
Recall that for $u\in Y_\gamma^{\le\lambda}$, the element $(\varpi^{-w_o(\lambda)},u^{-1}\gamma u)\in G_+(F)$ is generically regular semisimple. Hence it is in $V^{\mathrm{reg}}(\cO)$ if and only if its reduction mod $\varpi$ is in $V^{\mathrm{reg}}$. In the following, we will represent an element of $G_+(F)$ by its image under the embedding \eqref{embed-G+-equation} so that we can see its reduction mod $\varpi$ explicitly. In particular, if $u=\left[\begin{smallmatrix}1&x&z\\0&1&y\\0&0&1\end{smallmatrix}\right]\in U(F)$, then $(\varpi^{-w_o(\lambda)},u^{-1}\gamma u)\in G_+(F)$ will be denoted by the quadruple:
\[(\varpi^{-w_o(\lambda)},u^{-1}\gamma u)=\left(\varpi^2,\varpi^2,\begin{bmatrix}
 2\varpi^2&x\varpi^2&\frac{3z-xy}{2}\varpi^2\\0&\varpi^2&\frac{1}{2}y\varpi^2\\0&0&\varpi^2 
\end{bmatrix},\begin{bmatrix}\frac{1}{2}\varpi^2&0&0\\-\frac{1}{2}x\varpi^2 &\varpi^2&0\\
\frac{2xy-3z}{2}\varpi^2 & -y\varpi^2 & 2\varpi^2 \end{bmatrix}\right)\] 
If $u$ is in the generic open subset of $Y_1$, the reduction mod $\varpi$ of $(\varpi^{-w_o(\lambda)},u^{-1}\gamma u)$ will have the form
\begin{equation}\label{reduction-Y1-equation}
\left(0,0,\begin{bmatrix}
0&a&b\\0&0&0\\0&0&0
\end{bmatrix}, \begin{bmatrix}
0&0&0\\ -\frac{1}{2}a&0&0\\ c&0&0
\end{bmatrix}\right)
\end{equation}
where 
\[a = x\varpi^2(\mod\varpi),\quad b=\frac{3z-xy}{2}\varpi^2(\mod\varpi),\quad c=\frac{2xy-3z}{2}(\mod\varpi)\]
We assume $u$ is contained in the open subset of $Y_1$ defined by the condition that $a,b,c\in k^\times$ are mutually distinct nonzero elements.\par 
Now we calculate the centralizer of this element. An element $g\in G(k)$ is in this centralizer if and only if
\[g\begin{bmatrix}
0&a&b\\0&0&0\\0&0&0
\end{bmatrix}=\begin{bmatrix}
0&a&b\\0&0&0\\0&0&0
\end{bmatrix}g,\quad\text{and } ({}^tg)\begin{bmatrix}
0&0&0\\ -\frac{1}{2}a&0&0\\ c&0&0
\end{bmatrix}=\begin{bmatrix}
0&0&0\\ -\frac{1}{2}a&0&0\\ c&0&0
\end{bmatrix}({}^tg)\]
Write 
\[g=\begin{bmatrix}
a_{11}&a_{12}&a_{13}\\a_{21}&a_{22}&a_{23}\\a_{31}&a_{32}&a_{33}
\end{bmatrix}\quad\text{with }g^{-1}=\begin{bmatrix}
a_{22}a_{33}-a_{23}a_{32} & a_{13}a_{32}-a_{12}a_{33} & a_{12}a_{23}-a_{13}a_{22}\\
a_{23}a_{31}-a_{21}a_{33} & a_{11}a_{33}-a_{13}a_{31} & a_{13}a_{21}-a_{11}a_{23}\\
a_{21}a_{32}-a_{31}a_{22} & a_{12}a_{31}-a_{11}a_{32} & a_{11}a_{22}-a_{12}a_{21}
\end{bmatrix}\]
Then the above condition becomes the following equations
\begin{align*}
&a_{21}=a_{31}=0 \text{ (Here we used }a\ne0)\\ &aa_{11}-aa_{22}-ba_{32}=0\\ &aa_{11}-aa_{22}+2ca_{32}=0\\
&ba_{11}-aa_{23}-ba_{33}=0\\ &-2ca_{11}-aa_{23}+2ca_{33}=0
\end{align*}
Hence as long as $b+2c\ne0$, i.e. $xy-z\notin\varpi^{-1}\cO$, these equation will cuts out a subgroup of dimension $2$. Thus the generic open subset of $Y_1$ is contained in the regular locus.\par 
If $u$ is in the generic open subset of $Y_2$, the reduction mod $\varpi$ of $(\varpi^{-w_o(\lambda)},u^{-1}\gamma u)$ have the form
\begin{equation}\label{reduction-Y2-equation}
\left(0,0,\begin{bmatrix}
0&0&a\\0&0&b\\0&0&0
\end{bmatrix}, \begin{bmatrix}
0&0&0\\ 0&0&0\\ c&-2b&0
\end{bmatrix}\right)
\end{equation}
Calculations 
 similar to the previous case shows that the centralizer has dimension $2$, hence the generic open subset of $Y_2$ is also in the regular locus.\par 
Finally, if $u$ is in the generic open subset of $Y_3$, then the reduction mod $\varpi$ of  $(\varpi^{-w_o(\lambda)},u^{-1}\gamma u)$ will have the form
\begin{equation}\label{reduction-Y3-equation}
\left(0,0,\begin{bmatrix}
0&0&a\\0&0&0\\0&0&0
\end{bmatrix}, \begin{bmatrix}
0&0&0\\ 0&0&0\\ b&0&0
\end{bmatrix}\right)
\end{equation}
where $a,b\in k^\times$ are distinct nonzero elements. The centralizer of this element has dimension 4, hence irregular. Thus the generic open of $Y_3$ is irregular.
\end{proof}
Notice that in the above calculation, if $u$ is in an generic open subset of the components $Y_1$ or $Y_2$, then the reduction mod $\varpi$ of $(\varpi^{-w_o(\lambda)},u^{-1}\gamma u)$ lands in the nilpotent cone $\mathcal{N}^{\mathrm{reg}}$. Hence the components $Y_1, Y_2$ (which we call \emph{regular components}) corresponds to the 2 conjugacy classes in $\mathcal{N}^{\mathrm{reg}}$. Next we explicitly describe this correspondence.
\begin{prop}
The reduction mod $\varpi$ of a point in an open subset of $Y_1$ (rep. $Y_2$) lands in $\cN_1^{\mathrm{reg}}$ (resp. $\cN_2^{\mathrm{reg}}$).
\end{prop}
\begin{proof}
First consider a point $u\in Y_1\cap Y_\gamma^{\mathrm{reg}}$ and the reduction mod $\varpi$ of $(\varpi^{-w_o(\lambda)},u^{-1}\gamma u)$ as in \eqref{reduction-Y1-equation}. Using the notation in \eqref{reduction-Y1-equation}, we have
\[a = x\varpi^2(\mod\varpi),\quad b=\frac{3z-xy}{2}\varpi^2(\mod\varpi),\quad c=\frac{2xy-3z}{2}(\mod\varpi)\]
Recall from the calculation in previous proof that since $u$ is in the regular locus, we must have $b+2c\ne0$. Let 
\[a_{11}=(-ac-\frac{ab}{2})^{\frac{1}{3}}, a_{22}=ba_{11}^{-2}, a_{23}=-ca_{11}^{-2}, a_{32}=-aa_{11}^{-2}, a_{33}=-\frac{1}{2}aa_{11}^{-2}\] 
and 
\[g=\begin{bmatrix}
a_{11}&0&0\\0&a_{22}&a_{23}\\0&a_{32}&a_{33}
\end{bmatrix}\]
Then we check that $g\dot{w}_1h_0g^{-1}$ equals to the element \eqref{reduction-Y1-equation}.\par 
Thus for any point  $u\in Y_1\cap Y_\gamma^{\mathrm{reg}}$, the reduction mod $\varpi$ of $(\varpi^{-w_o(\lambda)},u^{-1}\gamma u)$ lands in $\cN_1^{\mathrm{reg}}$. Similar conclusion holds for $Y_2$ and $\cN_2^{\mathrm{reg}}$.
\end{proof}
Now we calculate the orbit of the centralizer $G_\gamma(F)=T(F)$. First we calculate the compact open subgroup of $G_\gamma(F)$ which acts trivially on $X_\gamma^{\le\lambda}$. Consider an element 
\[j=\begin{bmatrix}
t_1&0&0\\0&t_2&0\\0&0&t_3
\end{bmatrix}\in G_\gamma(F)  
\]
where $t_1t_2t_3=1$. For $u=\begin{bmatrix}
1 & x & z\\0&1&y\\0&0&1
\end{bmatrix}$ with $uK\in Y_\gamma^{\le\lambda}$, we would like to figure out for which $j\in G_\gamma(F)$, $juK=uK$. This is equivalent to $u^{-1}ju\in K$. We calculate that
\[u^{-1}\gamma u=\begin{bmatrix}
t_1&(t_1-t_2)x & zt_1-yt_2+(xy-z)t_3\\
0 & t_2 & y(t_2-t_3)\\
0&0&t_3
\end{bmatrix}\]
For which we see that if $j\equiv\mathrm{Id}\mod\varpi^2$, then this matrix has integral entries. Hence the congruence subgroup $1+\varpi^2\mathfrak{t}(\cO)$ acts trivially on $X_\gamma^{\le\lambda}$. Let $P^0_\gamma=T(\cO/\varpi^2\cO)$ be the $k$ group with $\dim P^0_\gamma=4$. Then $P_\gamma\cong P^0_\gamma\times X_*(T)$ is the quotient of $G_\gamma(F)$ by the congruence subgroup mentioned above. \par 
Note that if $u\in Y_\gamma^{\mathrm{reg}}$, then the stabilizer of $u$ in $P_\gamma$ is trivial. In particular, the open orbit has dimension equal to $\dim P_\gamma=4$, which is consistent with our calculation before. \par 
On the other hand, if $u$ is in the irregular locus, for example when $u\in Y_1$ and $xy-z\in\varpi^{-1}\cO$ or $u\in Y_3$, then there is a one-dimension subgroup of $P_\gamma$ which stabilizes $u$ and hence the orbit of $u$ will have dimension $3$.

\end{document}